\documentclass[12pt,reqno]{amsart}
\usepackage{cmap}
\usepackage[T1]{fontenc}
\usepackage[utf8]{inputenc}
\usepackage[english]{babel}
\usepackage{amsfonts,amssymb,amsthm,amsmath}
\usepackage{url}
\usepackage{enumerate}

\usepackage{secdot}

\usepackage{graphicx}
\usepackage{epstopdf}
\usepackage{a4wide}
\usepackage{xcolor}
\usepackage[colorlinks=true,linktocpage,pdfpagelabels,
bookmarksnumbered,bookmarksopen]{hyperref}
\definecolor{ForestGreen}{rgb}{0.1,0.6,0.05}
\definecolor{EgyptBlue}{rgb}{0.063,0.1,0.6}
\hypersetup{
	colorlinks=true,
	linkcolor=EgyptBlue,         
	citecolor=ForestGreen,
	urlcolor=olive
}

\usepackage[hyperpageref]{backref}

\newtheorem{theorem}{Theorem}
\newtheorem{proposition}[theorem]{Proposition}

\theoremstyle{definition}
\newtheorem{definition}[theorem]{Definition}
\newtheorem{remark}[theorem]{Remark}

\let\OLDthebibliography\thebibliography
\renewcommand\thebibliography[1]{
	\OLDthebibliography{#1}
	\setlength{\parskip}{1pt}
	\setlength{\itemsep}{1pt plus 0.3ex}
}

\numberwithin{equation}{section}
\numberwithin{theorem}{section}
\numberwithin{equation}{section}
\numberwithin{theorem}{section}

\newcommand{\rn}{\mathbb{R}^{N}}

\newtheorem{theo}{Theorem}[section]

\newtheorem{ex}[theo]{Example}
\numberwithin{equation}{section}

\makeatletter
\def\namedlabel#1#2{\begingroup
    #2%
    \def\@currentlabel{#2}%
    \phantomsection\label{#1}\endgroup
}
\makeatother

\title[Existence of solutions]{Existence of solutions of degenerate semilinear elliptic boundary value problems}

\author[R.N.~Dhara]{Raj Narayan DHARA}

\address[R.N.~Dhara]{\newline\indent
	Department of Mathematics and and Statistics,
	\newline\indent 
	Faculty of Science, Masaryk University, 
	\newline\indent 
	Kotl\'a{}\v{r}sk\'a 2, 611 37 Brno, Czech Republic
}
\email{rndhara@math.muni.cz,~rndhara@gmail.com}

\date{}


\thanks{
	RND acknowledges the support of the Czech Science Foundation, project GJ19-14413Y.
}

\begin{document}
\begin{abstract} 
	We show an existence of a weak solution of a degenerate and/or singular semilinear elliptic boundary value (nonhomogeneous) problem lying between a given weak subsolution and a given weak supersolution. It has been applied for an existence result of large solution to a similar problem.
\end{abstract} 
\maketitle 

\section{Introduction}
We consider the semilinear elliptic nonhomogeneous boundary value problems (BVP)
\begin{align}
-{\rm div}(w(x) \nabla u(x)) +b(x)f(u)=h(x) \ \text{in}\ \Omega,\label{eq:1}\\
u(x)=g(x)\ \ \text{on}\ \partial\Omega,\label{eq:2}
\end{align}
where $\Omega\subset \mathbb{R}^N,\ (N\ge 1)$ is an open bounded domain with smooth boundary $\partial\Omega$. 
We give an existence proof of solution to the degenerate and/or singular semilinear elliptic problem~\eqref{eq:1} with nonhomogeneous boundary value~\eqref{eq:2}.
Then it has been used to show an existence of {\em large solution} of a radially symmetric form of~\eqref{eq:1} with $h\equiv 0$. In general, we say~\eqref{eq:1} admits large solution when a solution $u(x)\to \infty$ as $x\to\partial\Omega$. One can show the existence of a nonnegative large (weak) solution in between sub- and supersolution, $\underline{u}\le u\le \bar{u}$, which by construction explode at the boundary. 
The equation~\eqref{eq:1} may have degeneracy ($w^{-1}$ is unbounded) and/ or singularity ($w$ is unbounded) in $\bar{\Omega}$.
Furthermore, these problems have many useful applications to study various problems arising in the theory of nonlinear diffusion processes and the system of those equations.

In order to look for solution, we adapt a theory to degenerate case, in contrast to which developed in~\cite{hess76,hess74} for nondegenerate case, so called, a penalty method instead of applying monotone iteration scheme for sub- supersolution method. 
  Though, unlike the classical case, results on regularity of the solution are not easily available when one deals with degenerate BVP in weighted function spaces. Further regularity assumptions on the boundary and on the involved functions which make the problem of consideration be solvable in some sense but it turns out to be much more restrictive in view of applications widely. Here we present an approach adapting from~\cite{hess76} to weighted settings which requires no regularity assumptions. On the other hand to treat the nonhomogeneous boundary conditions it requires a trace operator defined in weighted Sobolev space. The trace and extension theory in weighted Sobolev space is available (e.g.,~~\cite{tyulenev2014,mironescu2015,koskela2017}) upto some restrictions to the associated weight functions. An extension result with less restrictions on the weight functions, for example, can be found by Ka\l{}amajska and the author in~\cite{rndak2} and references therein.
 
For this purpose, we give definitions of weak solutions in weighted Sobolev spaces with a brief introduction of the admissible class of weights in Section~\ref{prelim}. We give a general proof of existence of solution of~\eqref{eq:1}-\eqref{eq:2}, given a sub- and supersolution, as a main result with Dirichlet boundary condition in Section~\ref{sec:main}. As an application of the previous results, in Section~\ref{sec:large}, we give an existence result for large solution for a radially symmetric problem.

\section{Solution, subsolution and supersolution}\label{prelim}
We assume the following hypotheses for $f,b$ and $w$ in~\eqref{eq:1}:
\begin{enumerate}
\item[\namedlabel{itm:h1}{(H1)}] $\|b/w\|_{\infty}<\infty$.
\item[\namedlabel{itm:h2}{(H2)}] The function $f:\mathbb{R}\to\mathbb{R}$ is continuous and monotone.
\end{enumerate}
We investigate the problem~\eqref{eq:1}-\eqref{eq:2} in weighted Sobolev space. 
We give the definitions of weighted Sobolev spaces and their properties needed in the present context.
 We use the notation $u^{+}:=\max(u,0),\ u^{-}:=\min(u,0)$.

\subsection{Weighted Sobolev spaces}
We define weighted Sobolev spaces following the definitions and proposition related to the involved weights.
Let $\mathcal{M}(\Omega)$ the set of all Borel  measurable functions. Elements of set
\(
W(\Omega):= \left\{w\in \mathcal{M}(\Omega):\ 0<w(x)<\infty,\ {\rm for\ a.e.}\  x\in\Omega \right\},
\)
 will be called positive weights on $\Omega$. 

\begin{definition}[Class $B_2(\Omega)$]
\rm
Let $p>1$. We will say that a weight $w\in W(\Omega)$ satisfies $B_2(\Omega)$-condition
(or simply, $w\in B_2$) if
\(
w^{-1}\in L^1_{{\rm loc}}(\Omega).
\)
\end{definition}
\noindent
The following simple observation results follow from H\"older inequality.

\begin{proposition}[\cite{kuf-opic}]\label{Lp:inclu:L1}
Let $w\in B_2$. Then
\(
L^{2}(\Omega;w)\) is the subset of \(  L^1_{{\rm loc}}(\Omega)
\).
\end{proposition}

\begin{definition}[Weighted  \textit{Sobolev space}]
\rm
Let $w\in W(\Omega)$ be the given weight. The linear set
\begin{equation}\label{polnorma}
\left\{ f\in L^{1}_{{\rm loc}}(\Omega) : f, \frac{\partial f}{\partial x_1} , \dots ,\frac{\partial f}{\partial x_n} \in L^2 (\Omega;w)\right\},
\end{equation}
where $\frac{\partial f}{\partial x_i}$ are distributional derivatives, equipped with the norm
\begin{equation}\label{norm:weight}
\| f\|_{W^{1,2}(\Omega;w)} := \|  f\|_{L^2 (\Omega;w)} +\| \nabla f\|_{L^2 (\Omega;w:\rn)},
\end{equation}
will be called weighted Sobolev space, denoted by $W^{1,2}(\Omega;w)$.
  Here $\|\cdot\|^{2}_{L^2 (\Omega;w)}:=\int_{\Omega}|\cdot|^{2}w\, dx$. Consistently we  choose the norm in the space $L^2(\Omega; w:\rn)$ as
$
\| (v_1,\dots,v_n)\|_{L^2(\Omega;w:\rn)}^{2}:= \sum_{i=1}^N\| v_i\|^2_{L^2(\Omega;w)}
$. Symbol $W^{1,2}_{0}(\Omega;w)$ will denote  the completion of  set ${C}^{\infty}_{0}(\Omega)$  in the norm of the space $W^{1,2}(\Omega;w)$. 
\end{definition}
The following theorem states formally what additional conditions on the weights make it a Banach space.

\begin{proposition}[\cite{kuf-opic}]\label{wlasn}
Let $w \in B_2$. Then linear set $W^{1,2}(\Omega;w)$ defined by \eqref{polnorma} equipped with the norm \eqref{norm:weight} is a Banach space.
In particular
  $W_{0}^{1,2}(\Omega;w)\subseteq \{ f\in L^{1}_{\textit{loc}}(\Omega) : f, \frac{\partial f}{\partial x_1} , \dots ,\frac{\partial f}{\partial x_n} \in L_{w}^2 (\Omega )\}\subseteq W^{1,1}_{loc}(\Omega)$.
\end{proposition}
Note that, the space $W^{1,2}_{0}(\Omega;w)$ is reflexive Banach space, see~ \cite[Remark 3.2, part 6)]{rndak3}.

We say a function $u\in W^{1,2}_{\rm loc}(\Omega;w)$ if and only if $u\in W^{1,2}(\Omega';w)$ for each open set $\Omega'$ whose closure is compact subset of $\Omega$.

The dual space of $W^{1,2}(\Omega;w)$ is denoted by $[W^{1,2}(\Omega;w)]^*=W^{-1,2}(\Omega;\rho)$, where $\rho=w^{-1}$ and $\langle\cdot ,\cdot\rangle$ represents the action between the duality pair. Note that if $w$ is a $2$-admissible weights (see e.g.~\cite[Section 1.1]{heino12}) then it also satisfies the conditions for the weights in Remark~\cite[Proposition 3.1]{rndak3}, and we have
\begin{align*}
W^{-1,2}(\Omega;\rho)=\left\{T=h_0 + {\rm div}\tilde{h}: h=(h_0,\tilde{h});\ h_0\in L^{2}(\Omega;\rho),\tilde{h}\in L^{2}(\Omega;\rho:\mathbb{R}^{n}) \right\},
\end{align*}
where $\tilde{h}=(h_1,h_2,\cdots,h_n)\in\mathbb{R}^{n}$. Furthermore, if the functions from the weighted Sobolev space $W^{1,2}(\Omega;w)$ admits Poincar\'e inequality then the first component $h_0$ in the above expression of the functional can be removed and vice-versa, see~\cite[Theorem 3.4]{rndak3}. 

Since our methodology primarily depends on the existence of an extension operator from the boundary of the given domain to the domain, we give some possible admissible class of weights for which an extension operator exists in weighted settings, see~\cite[Theorem 3.3]{rndak2}.
\begin{ex}\label{ilustracje}\rm
Let $w (x)=\tau ({\rm dist} (x,\partial\Omega))$. 
(i) $\tau\equiv 1$; (ii) $\tau(t) = t^{\alpha}$, $-1<\alpha <n-1$;
(iii)  $\tau(t)=t^{\alpha}\left(\ln\left(2+\frac{1}{t}\right)\right)^{\beta}$, $-1<\alpha <n-1,
\beta >0$;
(iv) $\tau (t)=(\log (2+ \frac{1}{t}))^{-\alpha}, \alpha >0$;
(v) $\tau (t)= 1-e^{\alpha t}$, $\alpha <0$, $n>2$.
\end{ex}
For other results on existence of an extension operator exists for Muckenhoupt class of weights, see e.g.~\cite{rndak2,koskela2017,mironescu2015,tyulenev2014} and references therein. In general, we can deal with such weight functions $w$ which either may vanish somewhere in $\overline{\Omega}$ or increase to infinity or both. 

We shall use the following definitions of solution, sub- and supersolution in weighted Sobolev spaces.
\begin{definition}[Boundary data]\label{def:boundary}
We say `$u= g$ on $\partial\Omega$' in the sense that $u-g\in W^{1,2}_{0}(\Omega;w)$ and `$u\ge g$ on $\partial\Omega$' in the sense that $(g-u)^{+}\in W^{1,2}_{0}(\Omega;w)$.
\end{definition}
\begin{definition}[Solutions]\label{def:wk:sol}
Let $w\in B_2$. Assuming~\ref{itm:h1}-\ref{itm:h2}, for a given $h\in W^{-1,2}(\Omega;\tau)$, a function $u\in W^{1,2}(\Omega;w)$ is a (weak) {\em supersolution (respectively, subsolution)} of~\eqref{eq:1}-\eqref{eq:2}
if $f(u)\in L^{2}(\Omega;w)$ and
\begin{align}
\int_{\Omega}\nabla u\cdot\nabla\varphi\, w(x)dx + \int_{\Omega} b(x)f(u) \varphi\, dx &\ge (\text{resp.},\ \le )\langle h(x) \varphi\rangle\quad\text{in}\ \Omega,\label{eq:3}\\
u
&\ge (\text{resp.} \le)\ g\qquad\text{on}\ \partial\Omega.\label{eq:3a}
\end{align}
for all $0 \le \varphi\in W^{1,2}_{0}(\Omega;w)$, which have compact support in $\Omega$.
We say a function $u\in W^{1,2}(\Omega;w)$ is a (weak) {\em solution} to~\eqref{eq:1}-\eqref{eq:2} if equalities hold in~\eqref{eq:3}-\eqref{eq:3a} for all $\varphi\in W^{1,2}_{0}(\Omega;w)$, which has compact support in $\Omega$. We say $u\in W^{1,2}(\Omega;w)$ is a {\em (weak) solution} of~\eqref{eq:1} on $\Omega$ if the equality in~\eqref{eq:3} holds only.

A function $u\in W^{1,2}_{\rm loc}(\Omega;w)$ is said to be a {\em local weak solution} to~\eqref{eq:1} in $\Omega$ if $u$ is a weak solution of~\eqref{eq:1} on $K$ for every compactly embedded subset $K\Subset\Omega$.
\end{definition}
\begin{definition}[Blow-up solution]\label{def:wk:large:sol}
A local weak solution $u$ of~\eqref{eq:1} is said to be a local weak {\em large (blow-up) solution} if $u$ is continuous on $\Omega$ and
\begin{align*}
u(x)\to\infty,\quad\text{as}\ x\to\partial\Omega.
\end{align*}
\end{definition}
\begin{remark}
We associate the principal linear part of~\eqref{eq:3} with a bilinear form,
$a(u,\varphi):=\int_{\Omega}\nabla u\cdot\nabla\varphi\, w(x)dx,\
$
which is then well-defined and bounded in $W^{1,2}_{0}(\Omega;w)\times W^{1,2}_{0}(\Omega;w)$. As a consequence of degenerate ellipticity condition, we have for some $c>0$,
\begin{align}\label{eq:degen:cond}
a(u,u)\ge c\|\nabla u\|^{2}_{L^{2}(\Omega;w)},\quad	\forall u\in W^{1,2}_{0}(\Omega;w).
\end{align}
\end{remark}

\section{Main result}\label{sec:main}

\begin{theorem}\label{thm:exis:sub:sup}
Suppose that $\underline{u},\ \bar{u}$ are weak sub- and supersolution (see Definition~\ref{def:wk:sol}) to problem~\eqref{eq:1}-\eqref{eq:2}, respectively, such that $\underline{u}\le \bar{u}$ a.e. in $\Omega$. Furthermore, assume that
\begin{align}\label{eq:bnd:f}
\int_{\Omega} \sup_{\underline{u}(x)\le t\le \bar{u}(x)}|f(t)|^2 w(x) \, dx<\infty.
\end{align}
Then problem~\eqref{eq:1}-\eqref{eq:2} has a weak solution $u$ with $\underline{u}\le u\le \bar{u}$ a.e. in $\Omega$. 
\end{theorem}
\begin{proof}
Let us denote by $\hat{g}\in W^{1,2}(\Omega;w)$ an extension of $g$ to $\Omega$ such that $\underline{u}\le \hat{g}\le \bar{u}$ a.e. in $\Omega$.
Then by definition of an extension operator, we have $\hat{g}_{|_{\partial\Omega}}=g$. Now the change of variable $u\mapsto u-\hat{g}$, reduce the problem to homogeneous boundary value problem. We can assume from now on that $\hat{g}=0$ and $\underline{u}\le 0\le \bar{u}$ a.e. in $\Omega$.

Let us define the modfied function $\tilde{f}$ as follows.
\begin{align}\label{eq:assum:f}
\tilde{f}(t):=
\begin{cases}
f(\underline{u}(x)),\quad t\le\underline{u}(x),\\
f(t),\quad \underline{u}(x)\le t\le \bar{u}(x),\\
f(\bar{u}(x))\quad t\ge \bar{u}(x),
\end{cases}
\end{align}
a.e. $x\in\Omega$ and for all $t\in\mathbb{R}$. We note that there exists a constant $c_0(\|b/w\|_{\infty},\|\tilde{f}\|_{L^{2}(\Omega;w)})=:c_0>0$ such that
\begin{align*}
\left|\int_{\Omega}b(x)\tilde{f}(u)v\, dx \right|\stackrel{\ref{itm:h1},\ \text{H\"older's},~\eqref{eq:bnd:f}}{\le} c_0\|v\|_{L^{2}(\Omega;w)}\quad\forall\ u,v\in L^{2}(\Omega;w).
\end{align*}
We consider the perturbed problem
\begin{align}\label{eq:per:1}
\tilde{a}(u,v)=\langle h,v\rangle,\qquad\forall\ v\in W^{1,2}_{0}(\Omega;w),
\end{align}
where
\begin{align*}
\tilde{a}(u,v)&:=a(u,v)+\int_{\Omega}b(x)\tilde{f}(u)v\, dx+\beta\int_{\Omega}\left((u-\underline{u})^{-}+(u-\bar{u})^{+}\right)v\, w(x)dx,\quad \beta>0,
\end{align*}
and show that~\eqref{eq:per:1} has a solution $u\in W^{1,2}_{0}(\Omega;w)$ which will essentially be a weak solution to our original problem~\eqref{eq:1}-\eqref{eq:2}.

The existence of solution of problem~\eqref{eq:per:1} will follow by employing a variant of the Minty Browder Theorem, see e.g.~\cite[Theorem 26.A, Corollary 27.19]{zaidler}, \cite[Theorem 2.1, p-171]{lions1969}, when the semilinear part $\tilde{a}$ is global coercive. For doing so, we note that
\begin{align*}
\left((t-\underline{u}(x))^{-}+(t-\bar{u}(x))^{+}\right)t\ge |t|^2 -\max\{|\underline{u}(x)|,|\bar{u}(x)|\}|t|,
\end{align*}
for any $t\in\mathbb{R}$ and fixed $x\in\Omega$.
Then the perturbed part, for some constant $c_1:=c_1(\|\max\{|\underline{u}(x)|,|\bar{u}(x)|\}\|_{L^2(\Omega;w)})$, can be estimated as follows.
\begin{align*}
\beta\int_{\Omega}\left((v-\underline{u})^{-}+(v-\bar{u})^{+}\right)v\, w(x)dx\ge \beta\|v\|^{2}_{L^{2}(\Omega;w)} -c_1\|v\|_{L^{2}(\Omega;w)},\quad\forall\ v\in W^{1,2}_{0}(\Omega;w).
\end{align*}
Consequently, using~\eqref{eq:degen:cond}, for $\tilde{c}:=\min\{c,\beta\}$, 
$ \tilde{a}(v,v)\ge \tilde{c} \|v\|^{2}_{W^{1,2}_{0}(\Omega;w)}-(c_0 +c_1)\|v\|_{L^{2}(\Omega;w)},\ \forall v\in W^{1,2}_{0}(\Omega;w),
$
and
\begin{align*}
\tilde{a}(v,v)/ \|v\|_{W^{1,2}_{0}(\Omega;w)}\rightarrow +\infty\qquad\text{as}\ \|v\|_{W^{1,2}_{0}(\Omega;w)}\to\infty.
\end{align*}
Hence, $\tilde{a}$ satisfying its global coerciveness condition, we have by~\cite[Corollary 27.19]{zaidler}, problem~\eqref{eq:per:1} has a solution. Now to complete the proof we only need to show that $\underline{u}\le u\le\bar{u}$ a.e. in $\Omega$, since then
$
\tilde{f}(u)=f(u).
$
Take the test function $v=(u-\underline{u})^-\in W^{1,2}_{0}(\Omega;w)$, see Definition~\ref{def:boundary}, and obtain
\begin{align}\label{eq:pert:2}
a(u,(u-\underline{u})^-)+\int_{\Omega}b(x)\tilde{f}(u)(u-\underline{u})^-\, dx
+\beta\|(u-\underline{u})^{-}\|_{L^{2}(\Omega;w)}^2~~~~~~~~~~~~~~~~~~ \nonumber\\
+\beta\int_{\Omega}(u-\bar{u})^{+}(u-\underline{u})^{-}\, w(x)dx
=\langle h,(u-\underline{u})^{-}\rangle.
\end{align}
Note that, since $\underline{u}\le\bar{u}$,
$\beta\int_{\Omega}(u-\bar{u})^{+}(u-\underline{u})^{-}\, w(x)dx=0.\
$
One can write using the bilinear property of $a(u,\cdot)$, to write
$a(u,(u-\underline{u})^-)=a(u-\underline{u},(u-\underline{u})^-)+a(\underline{u},(u-\underline{u})^-),\
$
and along with the fact of subsolution $\underline{u}$,
\begin{align*}
a(\underline{u},(u-\underline{u})^-)\ge \langle h, (u-\underline{u})^-\rangle - \int_{\Omega}b(x)f(\underline{u})(u-\underline{u})^-\, dx.
\end{align*}
From~\eqref{eq:pert:2}, we have
\begin{align}\label{eq:pert:3}
a(u-\underline{u},(u-\underline{u})^-)
+\int_{\Omega}b(x)(\tilde{f}(u)-f(u))(u-\underline{u})^-\, dx+\beta\|(u-\underline{u})^{-}\|_{L^{2}(\Omega;w)}^2 \le 0.
\end{align}
From the definition~\eqref{eq:assum:f} of $\tilde{f}$, we can see that
$\int_{\Omega}b(x)(\tilde{f}(u)-f(u))(u-\underline{u})^-\, dx=0.\
$
Also using the relation
$a((u-\underline{u}),(u-\underline{u})^-)=a((u-\underline{u})^-,(u-\underline{u})^-)\
$
to~\eqref{eq:pert:3}, we obtain
\begin{align*}
\|(u-\underline{u})^-\|^{2}_{W^{1,2}_{0}(\Omega;w)}\le 0.
\end{align*}
This implies that
$(u-\underline{u})^-=0\Rightarrow\ u(x)\ge \underline{u}(x)\quad \text{pointwise a.e. in}\ \Omega.\
$
Similarly, one can prove that
$(u-\bar{u})^+=0\Rightarrow\ u(x)\le \bar{u}(x)\quad \text{pointwise a.e. in}\ \Omega.\
$
Hence the proof of the theorem.
\end{proof}
\begin{remark}
 The above method can be applied to a boundary value problems associated with more general elliptic operator of the form: $Lu + p(\cdot, u, \nabla u)$, where $L$ is a second order quasilinear differential operator following a necessary modifications adapting from~\cite{hess74}.
\end{remark}
\begin{remark}
The above method may be applicable to a Neumann boundary value problem, with the boundary condition given as $w(x)\nabla u(x)\cdot n_{x}=q(x,u(x))$ and $n_x$ is the outward normal at the point $x\in\partial\Omega$ on $x\in\partial\Omega$, upto some given condition on $q$.
\end{remark}


\section{Large solutions}\label{sec:large}

Using Theorem~\ref{thm:exis:sub:sup} and following the techniques from~\cite{lopez2003} in our setting, we give an existence of a large solution to the problem~\eqref{eq:1}-\eqref{eq:2}.
\begin{theorem}\label{thm:exis}
Let $D$ be a bounded domain with smooth boundary, $w\in A_2$  and $\underline{u},\overline{u}\in W^{1,2}_{\rm loc}(D;w)\cap C(D)$ are the weak sub- and supersolution of~\eqref{eq:1} with $h\equiv 0$ according to Definition~\ref{def:wk:sol} in $D$ such that
$
\underline{u}\le \overline{u}\ \text{a.e. in}\ D.
$
Furthermore, we assume that the assumptions of Theorem~\ref{thm:exis:sub:sup} hold in each compact subsets of $D$
and 
\begin{align}\label{eq:infty:cond}
\lim_{\rm d(x,\partial D)\to 0}\underline{u}(x)=\infty\qquad\text{and}\qquad\lim_{\rm d(x,\partial D)\to 0}\overline{u}(x)=\infty,
\end{align}
Then \eqref{eq:1} possesses a (weak) solution $u\in W^{1,2}_{\rm loc}(D;w)\cap C(D)$ in between $\underline{u}$ and $\overline{u}$ such that $u\to\infty$ as ${\rm d(x)\to 0}$.
\end{theorem}
\begin{proof}
We consider
$D_n:=\left\{ x\in D:\ {\rm dist}(x,\partial D)>1/n,\ n\ge n_0\right\},
$
 where $n_0\ge 1$ should be chosen in such a way that $\partial D_n$ be a smooth boundary.
 Clearly, $D_n\subseteq D_{n+1}$ for $n\ge 1$  and
$\partial  D_n=\left\{ x\in  D:\ {\rm dist}(x,\partial D)=1/n\right\}\subset D.\
$
 Consider the problem
\begin{align}\label{eq:10}
\begin{cases}
-{\rm div}(w(x)\nabla u) + b(x)f(u) = 0\quad \text{in}\  D_n,\\
u =(\underline{u}+ \overline{u})/2\quad \text{on}\ \partial D_n,
\end{cases}
\end{align}
Clearly $\underline{u}$ and $\bar{u}$ are the sub- and supersolution of~\eqref{eq:10} as per Definition~\ref{def:wk:sol}. Then by Theorem~\ref{thm:exis:sub:sup}, the BVP~\eqref{eq:10} possesses a solution $u_n\in W^{1,2}( D_n;w)\cap C(\bar{ D}_{n})$ 
 such that $u_n - (\underline{u}+ \overline{u})/2\in W^{1,2}_{0}( D_n;w)$ and
\begin{align}\label{eq:11}
\underline{u}_{|_{ D_n}}\le u_n\le \overline{u}_{|_{ D_n}}\qquad\ \text{in}\  D_n.
\end{align}
Taking advantage of~\eqref{eq:11} 
together with Ascoli-Arzel\'a theorem gives rise to an existence of subsequence of $\{u_n\}_{n\ge n_0}$, say $\{u_{n_m}\}_{m\ge 1}$, for which
\begin{align*}
\lim_{m\to \infty}\|u_{n_m}-u_0\|_{C(\bar{ D}_{n_0})}=0,
\end{align*}
for some solution $u_0\in C(\bar{ D}_{n_0})$ of~\eqref{eq:10}. Let us consider the new restricted sequence $\{{u_{n_m}}{|_{ D_{n_1}}}\}_{m\ge 1}$ which again by the above reasoning, shows the existence of a subsequence $\{{u_{n_{m_l}}}\}_{l\ge 1}$ of $\{{u_{n_m}}{|_{ D_{n_1}}}\}_{m\ge 1}$, relabeled by $n_m$, such that, for some $u_1\in C(\bar{ D}_{n_1})$,
\begin{align*}
\lim_{l\to \infty}\|u_{n_{m_l}}-u_1\|_{C(\bar{ D}_{n_1})}=0.
\end{align*}
Note that, the continuity of the solution $u_1$ in $ D_{n_1}$, gives us $u_{1}|_{ D_{n_0}}=u_0$. Repeating this argument infinitely many times,  as a the pointwise limit of the diagonal sequence $\{ u_{m_m}\}$, we get the  required solution in between $\underline{u}$ and $\overline{u}$.
\end{proof}

We show by using Theorem~\ref{thm:exis} now the existence of the large solutions for two radially symmetric problems in a ball $B_{R}(x_0)\subset\rn$ centered at $x_0$ of radius $R$. We consider a typical weight $w(x):=d^{\alpha}(x),\ d(x)={\rm dist}(x,\partial B_{R}(x_0))=R-|x-x_0|$. 
\subsection{Asymptotic behavior of radially symmetric problems}
\begin{theorem}\label{thm:const}
Let $p>1,\ R>0, \gamma-\alpha\ge 0$ and consider the following singular problem
\begin{align}\label{eq:14}
\begin{cases}
\psi''+\left(\dfrac{N-1}{r}-\dfrac{\alpha}{R-r} \right)\psi'= a(r)(R-r)^{\gamma-\alpha}\psi^p \qquad \text{in}\ (0,R),\\
\psi'(0)=0,\quad \lim\limits_{r\to R} \psi(r) = +\infty,
\end{cases}
\end{align}
where $a\in C([0,R];(0,\infty))$. Then, for each $\varepsilon>0$ the problem~\eqref{eq:14} possesses a positive solution $\psi_{\varepsilon}$ such that
\begin{align}\label{eq:14a}
1-\varepsilon\le \liminf\limits_{r\to R}\dfrac{\psi_{\varepsilon}(r)}{K(R-r)^{-\beta}}
\le \limsup\limits_{r\to R}\dfrac{\psi_{\varepsilon}(r)}{K(R-r)^{-\beta}}
\le 1+\varepsilon,
\end{align}
where $K:= (\beta(\beta+1-\alpha)/a(R))^{1/(p-1)}$ and
\begin{align}\label{eq:beta}
\beta=\frac{2+\gamma-\alpha}{p-1}.
\end{align}
Therefore, for each $x_0\in \mathbb{R}^{N}$, the function
$ u_{\varepsilon}(x):=\psi_{\varepsilon}(r),\  r:= |x-x_0|,
$
gives us a radially symmetric nonnegative solution $u$ of
\begin{align}\label{eq:rad:1}
\begin{cases}
-{\rm div}(d^{\alpha}\nabla u)+a(|x-x_0|)d^{\gamma}(x)u^p =0 \qquad \text{in}\ B_R(x_0),\\
~~~~~~~~~~~~~~~~~~~~~~~~~~~~~~~u =\infty\qquad \text{on}\ \partial B_{R}(x_0),
\end{cases}
\end{align}
satisfying
\begin{align}\label{eq:rad:2}
1-\varepsilon\le \liminf\limits_{\rm d(x, \partial B_{R}(x_0))\to 0}\frac{u_{\varepsilon}(x)}{K d(x)^{-\beta}}\le \limsup\limits_{\rm d(x, \partial B_{R}(x_0))\to 0}\frac{u_{\varepsilon}(x)}{K d(x)^{-\beta}}\le 1+\varepsilon.
\end{align}
\end{theorem}
\begin{remark}
We see by  simple computations for a one dimensional degenerate or singular BVP
\begin{align}\label{eq:12}
\begin{cases}
((R-x)^{\alpha}u'(x))'=b(x)u^p,\quad x\in(0,R),\\
\lim\limits_{x\to R}u(x)=\infty,\quad
u'(0)=0,
\end{cases}
\end{align}
where $b(x)=B(x)(R-x)^{\gamma},\ B(R)\neq 0,\ \gamma-\alpha\ge 0$.
 Next we emply the following change of variables
\begin{align}
u(x)=(R-x)^{-\beta}\psi(x),\ x\in [0,R],
\end{align}
where $\beta>0$, to be determined, to the equation~\eqref{eq:12} and we have
\begin{align}\label{eq:13}
(R-x)^{-\beta+\alpha}\psi ''(x)+(2\beta-\alpha)(R-x)^{-\beta-1+\alpha}\psi'(x)~~~~~~~~~~~~~~~~~~~~~\nonumber \\ 
+ \beta(\beta+1-\alpha)(R-x)^{-\beta-2+\alpha}\psi(x)
=(R-x)^{\gamma-\beta p}\psi^{p}(x),
\end{align}
subject to the boundary conditions,
$
0<\psi(R)<\infty,
$  
so that exact blow-up rate of $u$ at $R$, will be given by $\beta$. Multiplying~\eqref{eq:13} by $(R-x)^{\beta+2-\alpha}$, we have
\begin{align*}
(R-x)^{2}\psi ''(x)+(2\beta-\alpha)(R-x)\psi'(x) + \beta(\beta+1-\alpha)\psi(x)=(R-x)^{\gamma-\beta p+\beta-\alpha+2}\psi^{p}(x).
\end{align*}
Since we are looking for the blow-up rate when $x\to R$, we assume
\begin{align*}
\lim\limits_{x\to R}(R-x)^{2}\psi ''(x)=0=\lim\limits_{x\to R}(R-x)\psi'(x).
\end{align*}
This necessarily gives us
$
\psi(R)=\left[\frac{\beta(\beta+1-\alpha)}{B(R)}\right]^{\frac{1}{p-1}}\quad\text{and}\quad 
\beta=\frac{2+\gamma-\alpha}{p-1}.
$
Now we derive the same blow up rate for the radial symmetric problem in many dimension.
\end{remark}
\begin{proof}[Proof of Theorem~\ref{thm:const}]
We shall construct a pair of sub- and supersolutions which essentially blow up on the boundary.\\
{\em Construction of supersolution:}
We claim that for each sufficiently small $\varepsilon>0$, there exists a constant $A_{\varepsilon}>0$ for which the function
$
\bar{\psi}_{\varepsilon}(r):= A+\bar{B}(r/R)^{2}(R-r)^{-\beta}
$
 provides us with a positive {\em supersolution} of~\eqref{eq:14} for each $A> A_{\varepsilon}$ if
\begin{align}\label{eq:16}
\bar{B} = (1+\varepsilon)
(\beta(\beta+1-\alpha)/a(R))^{1/(p-1)}.
\end{align}
Indeed, its positive in the sense that $\bar{\psi}_{\varepsilon}(0)=A>0,\  \bar{\psi}_{\varepsilon}'(r)\ge 0$ and $\lim\limits_{r\to R}\bar{\psi}_{\varepsilon}(r)=\infty$. 
Suppose that~\eqref{eq:beta} holds,
then we see that $\bar{\psi}_{\varepsilon}$ is supersolution to~\eqref{eq:14}, if and only if,
\begin{align}\label{eq:17}
\frac{2N\bar{B}}{R^2}(R-r)^2 + (3\beta+N\beta-2\alpha)\frac{\bar{B}}{R^2}r(R-r) + \bar{B}\beta(\beta+1-\alpha)\left(\frac{r}{R}\right)^{2}\nonumber\\
\le a(r)\left( A(R-r)^{\beta}+\bar{B}\left(\frac{r}{R}\right)^{2}\right)^{p}.
\end{align}
Note also that, at $r=R$, it turns out that
\begin{align*}
\bar{B}\ge (\beta(\beta+1-\alpha)/a(R))^{1/(p-1)}.
\end{align*}
Therefore, the choice of~\eqref{eq:16} and the continuity of $\bar{\psi}_{\varepsilon}$ upto $r=R$, the inequality~\eqref{eq:17} satisfied in $(R-\delta,R]$ for some $\delta=\delta(\varepsilon)>0$. Eventually, sufficiently large $A$ implies that the inequality~\eqref{eq:17} is satisfied in $[0,R]$. This shows one of the {\em construction of  supersolution}.

\noindent
{\em Construction of subsolution:}
Now we want to construct a subsolution to problem~\eqref{eq:14} by choosing, for each $\varepsilon>0$, there exists $C<0$ for which the function
\begin{align*}
\underline{\psi}_{\varepsilon}(r):= \max\left\{ 0, C+\underline{B}\left(\frac{r}{R}\right)^{2}(R-r)^{-\beta}\right\},
\end{align*}
shall provide us with nonnegative subsolution to~\eqref{eq:14} for which we need to determine a suitable choice of $\underline{B}$. 
For this we see that 
if 
\begin{align}\label{eq:19}
C+\underline{B}\left(\frac{r}{R}\right)^{2}(R-r)^{-\beta}\ge 0,
\end{align}
in some interval $(\bar{c},R)\subset (0,R)$, 
and suppose that~\eqref{eq:beta} hold then it is equivalently to say that $\underline{\psi}_{\varepsilon}$ is a subsolution if and only if
\begin{align}\label{eq:18}
\frac{2N\underline{B}}{R^2}(R-r)^2 + (3\beta+N\beta-2\alpha)\frac{\underline{B}}{R^2}r(R-r) + \underline{B}\beta(\beta+1-\alpha)\left(\frac{r}{R}\right)^{2}\nonumber\\
\ge
a(r) \left( C(R-r)^{\beta}+\underline{B}\left(\frac{r}{R}\right)^{2}\right)^{p}.
\end{align}
Since $f(r)$ in~\eqref{eq:19} is an nondecreasing function in $[0,R]$, indeed, $f'(r)\ge 0,\ r\in[0,R]$, we have for each $C<0$, there exists a constant $\bar{c}:=\bar{c}(C)\in (0,R)$ for which $f(r)<0$ if $r\in [0,\bar{c})$. On the other hand $f(r)\ge 0$ for $r\in [\bar{c},R)$. Moreover, $\bar{c}$ is decreasing as $C\to 0-$ and
\begin{align}\label{eq:20}
\lim\limits_{C\to-\infty}\bar{c}(C)=R,\ \lim\limits_{C\to 0}\bar{c}(C)=0. 
\end{align}
It is more than enough to have
\begin{align}\label{eq:19a}
\beta(\beta+1-\alpha)
\ge
a(r) \underline{B}^{p-1}\left(\frac{r}{R}\right)^{2(p-1)},\ \text{for each}\ r\in [\bar{c},R),
\end{align}
and since we have $C<0$,  the following condition
\begin{align*}
\underline{B}\beta(\beta+1-\alpha)\left(\frac{r}{R}\right)^{2}
\ge
a(r) \left( C(R-r)^{\beta}+\underline{B}\left(\frac{r}{R}\right)^{2}\right)^{p}, \ \text{for each}\ r\in [\bar{c},R)
\end{align*}
holds to satisfy~\eqref{eq:18}. So we can make a choice for $\underline{B}$ as follows
\begin{align}\label{eq:21}
\underline{B} = (1-\varepsilon)(\beta(\beta+1-\alpha)/a(R))^{1/(p-1)},
\end{align}
and there exists a constant for that, namely, $\tilde{\delta}=\tilde{\delta}(\varepsilon)>0$ for which~\eqref{eq:19a} is satisfied in $[R-\tilde{\delta},R)$. Moreover, because of~\eqref{eq:20}, there exists $C<0$ such that
$\bar{c}(C)=R-\delta(\varepsilon).$
For this choice of such $C$, it is proved that $\underline{\psi}_{\varepsilon}$ is a subsolution to~\eqref{eq:14}.
Together with the above growth conditions for sub- and supersolutions, i.e.
\begin{align*}
\lim\limits_{r\to R}\frac{\bar{\psi}_{\varepsilon}}{\bar{B}(R-r)^{-\beta}}=\lim\limits_{r\to R}\frac{\underline{\psi}_{\varepsilon}}{\underline{B}(R-r)^{-\beta}}=1,
\end{align*}
where $\bar{B},\ \underline{B}$ are the constants defined in~\eqref{eq:16} and~\eqref{eq:21} respectively, it follows from the construction of sub- and supersolution that it satisfy the conditions \eqref{eq:infty:cond} of Theorem~\ref{thm:exis}. Thereby we have an existence of large solution to~\eqref{eq:14}, denoted by $\psi_{\varepsilon}$, satisfying~\eqref{eq:14a}, and hence to~\eqref{eq:rad:1} satisfying~\eqref{eq:rad:2}.
\end{proof}


\bibliographystyle{plain}
\bibliography{../Large_Solutions/ref1}

\end{document}